\documentclass[a4paper,11pt,cleardoubleempty,bibtotoc]{article}

\long\def\forget#1{}

\usepackage{color}
\newcounter{commentcounter}

\def\?{\ 
{\bf\color{red}???}\ 
\immediate\write16{}
\immediate\write16{Warning: There was still a question mark . . . }
\immediate\write16{}}

\usepackage{hyperref}

\usepackage[T1]{fontenc}		
\usepackage[utf8]{inputenc}
\usepackage{amsfonts}
\usepackage{lmodern}
\usepackage{amsmath}					
\usepackage{enumerate}
\usepackage{amssymb}					
\usepackage{amsthm}
\usepackage{geometry}
\usepackage{mathrsfs}
\usepackage{stmaryrd} 
\usepackage{latexsym}
\usepackage[format=default,font=footnotesize,singlelinecheck=true,labelfont=bf]{caption} 
\usepackage{dsfont}
\usepackage[all]{xy}
\usepackage{scrpage2}
\usepackage{longtable}
\usepackage{appendix}
\usepackage[final]{graphicx}
\usepackage{pifont}
\usepackage{verbatim}
\usepackage{relsize}
\usepackage{paralist}

\usepackage{yhmath}

\newdir^{ (}{{}*!/-3pt/\dir^{(}}    
\newdir^{  }{{}*!/-3pt/\dir^{}}    
\newdir_{ (}{{}*!/-3pt/\dir_{(}}    
\newdir_{  }{{}*!/-3pt/\dir_{}}    

\newtheorem{Satz}{Satz}[section]		
\newtheorem{Lemma}[Satz]{Lemma} 

\newtheorem{Corollary}[Satz]{Corollary}
\newtheorem{Remark/Definition}[Satz]{Remark/Definition}

\newtheorem{Proposition}[Satz]{Proposition}

\newtheorem{Theorem}[Satz]{Theorem}

\theoremstyle{definition}
\newtheorem{Remark}[Satz]{Remark}
\newtheorem{Definition}[Satz]{Definition}
\newtheorem{Bemerkung/Definition}[Satz]{Bemerkung/Definition}

\newcommand{\boo}{\vphantom{A}}

\renewcommand{\phi}{\varphi}
\renewcommand{\theta}{\vartheta}

\renewcommand{\subset}{\subseteq}

\renewcommand{\:}{\colon}

\renewcommand{\r}{r}

\newcommand{\IF}{\mathbb{F}}

\newcommand{\IP}{\mathbb{P}}

\newcommand{\IZ}{\mathbb{Z}}
\newcommand{\bA}{\mathbf{A}} 
\newcommand{\bS}{\mathbf{S}} 
\newcommand{\bT}{\mathbf{T}}

\newcommand{\cA}{\mathcal{A}} 
 
\newcommand{\cC}{\mathcal{C}}

\newcommand{\cF}{\mathcal{F}}

\newcommand{\cM}{\mathcal{M}} 
 
\newcommand{\cO}{\mathcal{O}}

\newcommand{\fA}{\mathfrak{A}} 
\newcommand{\fa}{\mathfrak{a}} 
\newcommand{\fB}{\mathfrak{B}} 
\newcommand{\fC}{\mathfrak{C}} 

\newcommand{\fJ}{\mathfrak{J}} 

\newcommand{\fX}{\mathfrak{X}}

\newcommand{\uM}{{\underline{M\!}\,}}

\newcommand{\ucM}{\underline{\mathcal{M}}}
\newcommand{\uMhut}{{\underline{\hat M\!}\,}}
\newcommand{\ucMhut}{\underline{\hat{\mathcal M}}}

\newcommand{\ucN}{{\underline{\mathcal{N}\!}\,}}

\newcommand{\oF}{{\,\overline{\!F}}}

\newcommand{\Mhut}{\hat{M}}

\newcommand{\sigquer}{\bar{\sig}}
\newcommand{\piinv}{1/\pi}

\newcommand{\idm}{\mathfrak{m}}
\newcommand{\ida}{\mathfrak{a}}
\newcommand{\idb}{\mathfrak{b}}

\newcommand{\eps}{\varepsilon}
\newcommand{\sig}{\sigma}

\newcommand{\jac}{\mathfrak{j}}

\newcommand{\inv}{{-1}}

\newcommand{\pl}{\underleftarrow{\lim}}

\newcommand{\epspi}{(\eps,\pi)}

\newcommand{\pfeil}{\rightarrow}
\newcommand{\gehtauf}{\mapsto}
\newcommand{\haken}{\hookrightarrow}
\newcommand{\lb}{\llbracket}
\newcommand{\rb}{\rrbracket}
\newcommand{\lbc}{(\!(} 
\newcommand{\rbc}{)\!)}

\newcommand{\la}{\langle}
\newcommand{\ra}{\rangle}
\def\isoto{\stackrel{}{\mbox{\hspace{1mm}\raisebox{+1.4mm}{$\scriptstyle\sim$}\hspace{-3.5mm}$\longrightarrow$}}}

\DeclareMathOperator{\Spec}{Spec}

\newcommand{\Frac}{\text{\textup{Frac}}}
\newcommand{\alg}{\text{\textup{alg}}}

\newcommand{\Gal}{\text{\textup{Gal}}}

\newcommand{\GL}{\text{\textup{Gl}}}
\newcommand{\id}{\text{\textup{id}}}

\newcommand{\coker}{\text{\textup{coker}}}
\newcommand{\an}{\text{\textup{an}}}

\newcommand{\Frob}{\text{\textup{Frob}}}

\newcommand{\Hom}{\text{\textup{Hom}}}

\newcommand{\im}{\text{\textup{im}}}

\newcommand{\Spf}{\text{\textup{Spf}}}
\newcommand{\Spm}{\text{\textup{Sp}}}

\newcommand{\cha}{\text{\textup{char}}}

\newcommand{\rk}{\text{\textup{rk}}}
\newcommand{\rig}{\text{\textup{rig}}}
\newcommand{\ord}{\text{\textup{ord}}}
\newcommand{\Diag}{\text{\textup{Diag}}}

\newcommand{\FMod}{\text{\textup{FMod}}}

\title{A criterion for good reduction of Drinfeld modules \\ and Anderson motives in terms of local shtukas}
\author{U.~Hartl, S.~H\"usken}
\date{}

\begin{document}

\maketitle

\vspace{-1cm}
\begin{abstract}
\noindent
For an Anderson $A$-motive over a discretely valued field whose residue field has $A$-characteristic $\eps$, we prove a criterion for good reduction in terms of its associated local shtuka at $\eps$. This yields a criterion for good reduction of Drinfeld modules. Our criterion is the function-field analog of Grothendieck's~\cite[Proposition~IX.5.13]{SGA7} and de Jong's~\cite[2.5]{dJ98} criterion for good reduction of an abelian variety %$\cA$ 
over a discretely valued field with residue %field of 
characteristic $p$ in terms of its associated $p$-divisible group% $\cA[p^{-\infty}]$ of $\cA$
.

\noindent
{\it Mathematics Subject Classification (2000)\/}: 
%11F80,  % Galois representations (Discontinuous groups and automorphic forms)
11G09,  % Drinfeld Modules, higher dimensional motives
%11S20,  % Galois theory of local and $p$-adic fields
%11S25,  % Galois cohomology of local and $p$-adic fields
%13A35,  % Characteristic $p$ methods (Frobenius endomorphism) ...
%14F30,  % $p$-adic cohomology, crystalline cohomology
%14G20,  % Local ground fields
%14G22,  % Rigid analytic geometry
%14G35,  % Modular and Shimura varieties
%11G18,  % Arithmetic aspects of modular and Shimura varieties
(14L05)  % Formal groups, $p$-divisible groups
%14M15)  % Grassmannians, Schubert varieties, flag manifolds
%20G25,  % Linear algebraic groups over local fields and their integers
\end{abstract}

\tableofcontents

\section{Introduction}

We fix a finite field $\IF$ with $\r$ elements and characteristic $p$. Let $\cC$ be a smooth projective and geometrically irreducible curve over $\IF$ with function field $Q=\IF(\cC)$. Let $\infty\in \cC$ be a closed point and let $A=\Gamma(\cC\smallsetminus\{\infty\},\cO_\cC)$ be the $\IF$-algebra of those rational functions on $\cC$ which are regular outside $\infty$. For every $\IF$-algebra $R$ we let $\sig$ be the endomorphism of $A_R:=A\otimes_\IF R$ given by $\sig:=\id_A\otimes{\rm Frob}_{\r,R}\:a\otimes b\gehtauf a\otimes b^\r$ for $a\in A$ and $b\in R$.

Let $o_L$ be a complete discrete valuation ring containing $\IF$, with fraction field $L$, uniformizing parameter $\pi$, maximal ideal $\idm_L=(\pi)$ and residue field $\ell=o_L/\idm_L$. We assume that $\ell$ is a finite field extension of $\ell^p$. This is equivalent to saying that $\ell$ has a finite $p$-basis over $\ell^p$ in the sense of \cite[\S\,V.13, Definition~1]{BourbakiA47}. It holds for example if $\ell$ is perfect, or if $\ell$ is a finitely generated field. Since every Anderson $A$-motive over $L$ can be defined over a finitely generated subfield of $L$ our restriction on $\ell$ is not serious. Let $c^*\:A\to o_L$ be a homomorphism of $\IF$-algebras such that the kernel of the composition $A\pfeil o_L\twoheadrightarrow\ell$ is a \emph{maximal} ideal $\eps$ in $A$. We say that \emph{the residue field $\ell$ has finite $A$-characteristic $\eps$}. We do not assume that $c^*\:A\to o_L$ is injective. So $L$ can have either generic $A$-characteristic $\ker c^*=(0)$ or finite $A$-characteristic $\ker c^*=\eps$. In the following we will consider various $A_{o_L}$-algebras. In all of them we consider the ideal generated by $\{a\otimes 1-1\otimes c^*(a)\: a\in A\}\subset A_{o_L}$. By abuse of notation we denote all these ideals by $\fJ$.

By an \emph{Anderson $A$-motive over $L$} we mean a pair $\uM=(M,F_M)$ consisting of a locally free $A_L$-module $M$ of finite rank, and an injective $A_L$-homomorphism $F_M\:\sig^*M\pfeil M$ where $\sig^*M:=M\otimes_{A_L,\sig}A_L$, such that $\coker(F_M)$ is a finite dimensional $L$-vector space and is annihilated by a power of $\fJ$. We say that $\uM$ has \emph{good reduction over $o_L$} if there exists a locally free $A_{o_L}$-module $\cM$ and an injective $A_{o_L}$-homomorphism $F_\cM\:\sig^*\cM\pfeil\cM$ such that $(\cM,F_\cM)\otimes_{A_{o_L}}A_L\cong\uM$ and $\coker(F_\cM)$ is a finite free $o_L$-module which is annihilated by a power of $\fJ$. We call $\ucM=(\cM,F_{\cM})$ a \emph{good model of $\uM$}. 
In particular if $\uM=\uM(\phi)$ is the Anderson $A$-motive associated with a Drinfeld $A$-module $\phi$ over $L$, then $\uM$ has good reduction if and only if $\phi$ has good reduction; see Proposition~\ref{PropGoodRedDriMod}.

Anderson $A$-motives are function-field analogs of abelian varieties. For an abelian variety $\cA$ over a discretely valued field $K$ with residue field of characteristic $p$ there are criteria for good reduction in terms of local data. For a prime number $l\ne p$ the criterion of N\'eron-Ogg-Shavarevich \cite[\S1, Theorem 1]{Serre-Tate} states that $\cA$ has good reduction if and only if the $l$-adic Tate module $T_l\cA$ of $\cA$ is unramified as a $\Gal(K^\alg/K)$-representation. At the prime $p$ the criterion of Grothendieck \cite[Proposition~IX.5.13]{SGA7} (for $\cha(K)=0$), respectively de Jong~\cite[2.5]{dJ98} (for $\cha(K)=p$) states that $\cA$ has good reduction if and only if the Barsotti-Tate group $\cA[p^\infty]$ has good reduction.

These criteria have function-field analogs for Anderson $A$-motives. The analog of the N\'eron-Ogg-Shavarevich-criterion was proved by Gardeyn~\cite[Theorem 1.1]{Gardeyn2}. In this article we simultaneously prove the analog of Grothendieck's and de Jong's criterion. Here the function-field analogs of Barsotti-Tate groups are local shtukas \cite[\S\,2.1]{HHabil} which are defined as follows. Let $A_{o_L,\epspi}$ be the $\epspi$-adic completion of $A_{o_L}$. An (\emph{effective}) \emph{local shtuka at $\eps$ over $o_L$} is a pair $\uMhut=(\Mhut,F_{\Mhut})$ consisting of a finite free $A_{o_L,\epspi}$-module $\Mhut$ and an injective $A_{o_L,\epspi}$-homomorphism $F_{\Mhut}\:\sig^*\Mhut\pfeil\Mhut$ such that $\coker(F_{\Mhut})$ is a finite free $o_L$-module and is annihilated by a power of $\fJ$. The local shtuka associated with a good model $\ucM$ of an Anderson $A$-motive is $\uMhut(\ucM):=\ucM\otimes_{A_{o_L}}A_{o_L,\epspi}$. Strictly speaking effective local shtukas are the function field analogs of the $F$-crystals of Barsotti-Tate groups. The analogs of the latter are called \emph{$\eps$-divisible local Anderson-modules} and their category is equivalent to the category of effective local shtukas; see \cite{HS} for more details. Our analog of Grothendieck's and de Jong's reduction criterion is now the following

\medskip\noindent
{\bfseries Corollary~\ref{CorGoodRedCrit}.}
{\it Let $\uM$ be an Anderson $A$-motive over $L$. Then $\uM$ has good reduction over $o_L$ if and only if there is an effective local shtuka $\uMhut$ at $\eps$ over $o_L$ and an isomorphism $\uM \otimes_{A_L} A_{o_L,\epspi}[\piinv]\cong \uMhut\otimes_{A_{o_L,\epspi}}A_{o_L,\epspi}[\piinv]$.
}

\medskip

(In the body of the text we prove a slightly stronger statement.) This applies in particular if $\uM$ is the Anderson $A$-motive associated with a Drinfeld module $\phi$ over $L$ to give a criterion for good reduction of $\phi$ in terms of its associated local shtuka. The reformulation of this criterion in terms of the $\eps$-divisible local Anderson-module of $\phi$ is given in \cite{HS}.

\medskip

{\bfseries Acknowledgements.} We would like to thank the anonymous referee for his careful reading and for asking an interesting question which lead to the answer given in Remark~\ref{RemNotEquiv}. We also thank the Deutsche Forschungsgemeinschaft for supporting this research in form of SFB 878.

\section{The base rings}\label{The base rings}

Let $o_L$ be an equi-characteristic complete discrete valuation ring containing the finite field $\IF$, with quotient field $L=\Frac(o_L)$ and residue field $\ell=o_L/\idm_L$, where $\idm_L\subset o_L$ is the maximal ideal of $o_L$. We assume that $\ell$ is a \emph{finite} field extension of $\ell^p:=\{b^{\,p}\:b\in\ell\}$. We fix a uniformizer $\pi=\pi_L$ of $o_L$ and sometimes identify $o_L$ with $\ell\lb\pi\rb$. Let $v=v_\pi=\ord_\pi(\cdot)$ be the discrete valuation on $L$ normalized by $v(\pi)=1$. 

We assume that there is an $o_L$-valued point $c\in \cC(o_L)$ such that the corresponding $\IF$-morphism $c\:\Spec(o_L)\pfeil \cC$ factors via $\cC\smallsetminus\{\infty\}\subset \cC$. Such a datum corresponds to a homomorphism of $\IF$-algebras $c^*\: A\pfeil o_L$ which we call the \emph{characteristic map}. We further assume that the closed point $V(\pi)\subset \Spec(o_L)$ is mapped to a closed point $\eps$ of $\Spec(A)\subset \cC$. The latter is the kernel of the composition $A\pfeil o_L\twoheadrightarrow \ell$. So, in accordance with Drinfeld's terminology \cite{DrinfeldEM}, we call $\eps$ the \emph{residue characteristic} or \emph{residual characteristic place of $Q$}. 
By continuity, the characteristic map $c^*\: A\pfeil o_L$ factors through a morphism of complete discrete valuation rings $A_\eps\pfeil o_L$ where $A_\eps$ is the completion of $A$ at the characteristic place $\eps$. Note that $A_\eps\pfeil o_L$ is injective if $c^*$ is injective, and factors through $A/\eps$ if $c^*$ is not injective.

\begin{Remark}\label{t}
Since $A$ is a Dedekind domain there is a power $\eps^m$ which is a principal ideal in $A$. We fix a generator $t$ of $\eps^m$ and frequently use the finite flat monomorphism of $\IF$-algebras $\iota\:\IF[z]\pfeil A, z\mapsto t$.
\end{Remark}

For any $\IF$-algebra $R$ we abbreviate $A_R:=A\otimes_\IF R$. In particular, $A_{o_L}\subset A_L$ is a noetherian integral domain, and by virtue of the equality $A_\ell\cong  A_{o_L}/\pi A_{o_L}$ it follows that $\pi\in o_L$ is a prime element of $A_{o_L}$. 

\begin{Definition} Let $A_{o_L,\pi}$ (resp., $A_{o_L,\epspi}$) be the completion of the $o_L$-algebra $A_{o_L}$ for the $\pi$-adic topology (resp., the $(\eps,\pi)$-adic topology).
\end{Definition}

By Krull's Theorem (\cite{BourbakiCA}, III.3.2), the ring $A_{o_L}$ is separated for both the $\pi$-adic and the $\epspi$-adic topology. 
The topological $o_L$-algebra $A_{o_L,\pi}$ is admissible in the sense of Raynaud, i.e.\ it is of topologically finite presentation and has no $\pi$-torsion. In particular, the $L$-algebra $A_{o_L,\pi}[\piinv]$ is affinoid in the sense of rigid analytic geometry; see \cite{SFB378,FRG1,BGR}.

For example if $\cC=\IP_\IF^1$ and $A=\IF[z]$ then we have $A_{o_L}= o_L[z]$ and correspondingly $A_L= L[z]$. Let us specify that $\eps=z\IF[z]$. Our choice of a uniformizer $\pi$ gives rise to an identification $o_L=\ell\lb \pi\rb$. Consequently $o_L\lb z\rb=\ell\lb\pi\rb\lb z\rb=\ell\lb \pi,z\rb=A_{o_L,\epspi}$.
On the other hand, the $\pi$-adic completion of $o_L[z]$ equals $o_L\la z\ra:=\{\sum\limits_{i=0}^\infty b_iz^i\: v(b_i)\pfeil\infty (i\pfeil\infty)\}$, and since $L\la z\ra=o_L\la z\ra\otimes_{o_L}L$, we may view $A_{o_L,\pi}[\piinv]$ as a replacement, for general $\cC$, of the Tate algebra $L\la z\ra$ of strictly convergent power series in one indeterminate $z$ over $L$, which serves as coordinate ring for the one-dimensional affinoid unit ball in rigid analytic geometry. 

There is a natural embedding $A_L\pfeil A_{o_L,\pi}[\piinv]$ which, for general $\cC$, replaces the completion homomorphism $L[z]\pfeil L\la z\ra$, and which itself can be regarded as a completion map with respect to the $L$-algebra norm-topology on the \emph{reduced} affinoid $L$-algebra $A_{o_L,\pi}[\piinv]$ and its restriction on $A_L$; see \cite[\S1.4, Proposition~19]{SFB378}. 
Note that the canonical homomorphism $A_{o_L}\pfeil A_{o_L,\epspi}$ factors uniquely via $A_{o_L,\pi}$, where the induced map $A_{o_L,\pi}\pfeil A_{o_L,\epspi}$ identifies $A_{o_L,\epspi}$ with the $\epspi A_{o_L,\pi}$-adic completion of $A_{o_L,\pi}$.
Since $A_{o_L,\pi}$ is a regular integral domain, it is $\epspi A_{o_L,\pi}$-adically separated by Krull's theorem and $A_{o_L,\pi}\pfeil A_{o_L,\epspi}$ is injective and flat.

Recall that there is a finite flat monomorphism of $\IF$-algebras $\iota\: \IF[z]\pfeil A$ which identifies the indeterminate $z$ with the generator $t\in A$ of $\eps^m$ chosen in Remark~\ref{t}. The $o_L$-algebra homomorphism $\iota\otimes \id\: o_L[z]\pfeil A_{o_L}$, $\sum_\nu a_\nu z^\nu\gehtauf \sum_\nu t^\nu\otimes a_\nu$, is finite flat, so that we obtain finite flat maps
\begin{equation}\label{CProj-Abb}
o_L\la z\ra\pfeil A_{o_L,\pi},\quad L\la z\ra\pfeil A_{o_L,\pi}[\piinv],\quad o_L\lb z\rb\pfeil A_{o_L,(t,\pi)}, \quad \ell[z]\pfeil A_\ell.
\end{equation}
Here the $(t,\pi)$-adic completion  $A_{o_L,(t,\pi)}$ of $A_{o_L}$ equals $A_{o_L,\epspi}$ since $(\eps,\pi)^m\subset(\eps^m,\pi)=(t,\pi)$ in $A_{o_L}$.

\begin{Lemma}
If $A_{o_L,\eps}$ denotes the $\eps$-adic completion of $A_{o_L}$, the canonical map $A_{o_L,\eps}\pfeil A_{o_L,\epspi}$ is an isomorphism. \qed
\end{Lemma}

\section{Frobenius modules}\label{Frobenius}

The $\r$-Frobenius $\Frob_{\r}\: o_L\pfeil o_L$, $x\gehtauf x^\r$, gives rise to an endomorphism $$\sig=\id_A\otimes\Frob_{\r}\: A_{o_L}\pfeil A_{o_L},\quad a\otimes x\gehtauf a\otimes x^{\r},$$
which extends to give a map $\id_A\otimes \Frob_{\r,L}\: A_L\pfeil A_L$ again denoted by $\sig$. On the other hand, reducing mod $\pi$ gives $\sigquer=\id_A\otimes\Frob_{\r,\ell}\: A_\ell\pfeil A_\ell$. The latter is a finite flat endomorphism of the Dedekind domain $A_\ell$, because $\ell$ is finite over $\ell^p$. 
The map $\sig\: A_{o_L}\pfeil A_{o_L}$ is $\pi$-adically and $\epspi$-adically continuous and therefore extends to give endomorphisms $A_{o_L,\pi}\pfeil A_{o_L,\pi}$ and $A_{o_L,\epspi}\pfeil A_{o_L,\epspi}$, again denoted by $\sig$. 

\begin{Lemma}\label{Frob-cocart}
In the commutative diagram
$$\xymatrix{A_{o_L} \ar[r] \ar[d]_\sig & A_{o_L,\pi} \ar[r] \ar[d]_\sig & A_{o_L,\epspi} \ar[d]_\sig\\
A_{o_L} \ar[r] & A_{o_L,\pi} \ar[r] & A_{o_L,\epspi} }$$
both squares are cocartesian, and the vertical arrows are finite flat. 
\end{Lemma}

We let the proof be preceded by the following

\medskip
\noindent
\textit{Remark}. Via the identification $o_L=\ell\lb\pi\rb$, the $\r$-Frobenius $\Frob_{\r,o_L}\: o_L\pfeil o_L$ is mirrored by the map $\ell\lb\pi\rb\pfeil \ell\lb\pi\rb$, $\sum\limits_{\nu=0}^{\infty}a_\nu \pi^\nu\gehtauf\sum\limits_{\nu=0}^{\infty}a_{\nu}^\r\pi^{\r\nu}$. Choosing an $\ell^r$-basis of $\ell$ and lifting it to a subset $W$ of $o_L$, this implies $(\Frob_{\r,o_L})_*o_L=\bigoplus\limits_{i=0}^{\r-1}\bigoplus\limits_{w\in W}o_L\,w\pi^i$, so that $\Frob_{\r,o_L}\: o_L\pfeil o_L$ is finite flat.

\begin{proof}[Proof of Lemma \ref{Frob-cocart}]
By base change the remark implies that $\sig=\id_A\otimes\Frob_{\r,o_L}\: A_{o_L}\pfeil A_{o_L}$ is finite flat, and that $A_{o_L}\otimes_{\sig,A_{o_L}}A_{o_L,\pi}$ is a finite flat $A_{o_L,\pi}$-module and hence equals the $\pi$-adic completion of the $A_{o_L}$-module $\sig_*A_{o_L}$. If we let $\ida=\sig(\pi A_{o_L})A_{o_L}=\pi^\r A_{o_L}$ and $\idb=\pi A_{o_L}$, we get $\idb^\r=\ida\subset\idb$. Consequently, by \cite[Lemma 7.14]{EisCA}, the inverse systems $(A_{o_L}/\ida^n)_n$ and $(A_{o_L}/\idb^n)_n$ give the same limit, which shows that the square on the left is cocartesian, and that $\sig\: A_{o_L,\pi}\pfeil A_{o_L,\pi}$ is finite flat. Similarly, we have $\sig\epspi A_{o_L}=(\eps,\pi^\r)\subset\epspi$ as well as $\epspi^\r\subset (\eps,\pi^\r)$, which proves that the displayed diagram qualifies $A_{o_L,\epspi}$ as tensor product $A_{o_L,\epspi}\otimes_{A_{o_L},\sig}A_{o_L}$, and that $\sig\: A_{o_L,\epspi}\pfeil A_{o_L,\epspi}$ is finite flat. 
\end{proof}

Finally, note that the embedding of $o_L$-algebras $\iota\otimes\id \: o_L[z]\pfeil A_{o_L}$ commutes with $\sig\: A_{o_L}\pfeil A_{o_L}$ and the $\r$-Frobenius lift of $o_L[z]$, given by $o_L[z]\pfeil o_L[z],\; \sum_\nu a_\nu z^\nu\gehtauf\sum_\nu a_\nu^\r z^\nu.$
Consequently, also the embeddings from \eqref{CProj-Abb} are Frobenius-equivariant.

\bigskip

Let $B$ be an $o_L$-algebra together with a ring endomorphism $\sig\:B\pfeil B$ such that $\sig$ and $\Frob_{\r,o_L}\: o_L\pfeil o_L$ are compatible with the structure map $o_L\pfeil B$. For example, $B$ could be any of the base rings considered above.
\begin{Definition}
We define the category $\FMod(B)$ of \emph{Frobenius $B$-modules} (or simply \emph{$F$-modules} over $B$) as follows: 
\begin{itemize}
\item An object of $\FMod(B)$ is a pair $\uM=(M,F)$ consisting of a $B$-module $M$ which is locally free of finite rank, together with an \emph{injective} $B$-linear map $F=F_{M}\:\sig^*M\pfeil M$, where $\sig^*M:=M\otimes_{B,\sig}B$. 
\item A \emph{morphism} of Frobenius $B$-modules $(M,F_{M})\pfeil (N,F_{N})$ is a $B$-linear map $\phi\:M\pfeil N$ between the underlying $B$-modules such that $\phi$ is \emph{$F$-equivariant}, i.e.\ such that $\phi\circ F_{M}=F_{N}\circ \sig^*\phi$. 
It is called an \emph{isomorphism} if $\phi$ is an isomorphism of the underlying $B$-modules. 
\end{itemize}
\end{Definition}

Let $B'$ be a flat $B$-algebra together with a ring endomorphism $\sig\: B'\pfeil B'$ extending the Frobenius lift of $B$, as explained before. Then the exact functor $\cdot\otimes_{B}B'$ from $B$-modules to $B'$-modules yields a functor $\FMod(B)\pfeil \FMod(B')$. If the structure map $B\pfeil B'$ is, in addition, injective then the induced functor on $\FMod(B)$ is faithful since, given a map $f\:  M\pfeil  N$ of finite projective $B$-modules, restricting its image $f\otimes\id\:  M\otimes_{B}B'\pfeil  N\otimes_{B}B'$ to $ M$ gives back $f$. In particular, we obtain a natural commutative diagram of categories and faithful functors
$$\xymatrix{\FMod(A_{o_L}) \ar[r] \ar[d] & \FMod(A_{o_L,\pi}) \ar[d] \ar[r] & \FMod(A_{o_L,\epspi})\ar[d] \\
\FMod(A_L) \ar[r] & \FMod(A_{o_L,\pi}[\piinv]) \ar[r] & \FMod(A_{o_L,\epspi}[\piinv])}$$

Slightly abusing notation, we agree to write $\uM\otimes_B B'$ for $(M\otimes_B B', F_M\otimes \id_{B'})$, whenever $\uM=(M,F_M)$. 

\section{Anderson motives}\label{AAM}

Let $\fJ\subset A_{o_L}$ be the ideal generated by $a\otimes 1- 1\otimes c^*(a)$ for all $a\in A$. For example, if $\cC=\IP_\IF^1$ and $A=\IF[z]$, then  $\fJ=(z-\zeta)\subset o_L[z]$ where $\zeta=c^*(z)$. Note that the convention introduced in Remark~\ref{t} that $(z)=\eps^m$ implies $\zeta\in\idm_L$. So $\zeta=0$ if $c^*$ is not injective. By abuse of notation we denote the ideal generated by $\fJ$ in any $A_{o_L}$-algebra again by $\fJ$.
We consider the following variant of Anderson's~\cite{Anderson} $t$-motives.

\begin{Definition}
An \emph{Anderson $A$-motive over $L$} is an object $\uM=(M,F_M)\in\FMod(A_L)$ such that $\coker(F_{ M})$ is a finite-dimensional $L$-vector space and is annihilated by a power of $\fJ$. A \emph{morphism} of Anderson $A$-motives is defined as a morphism inside $\FMod(A_L)$.
\end{Definition}

Since $\Spec(A_L)$ is of finite type over $L$, one can consider its rigid analytification $\Spec(A_L)^\an$; see \cite{SFB378}, \cite{BGR}, \cite{FvdP}. In accordance with \cite{BH}, we denote this rigid analytic $L$-space by $\fA(\infty)$. On the other hand, the formal completion of the $o_L$-scheme $X=\Spec(A_{o_L})$ along its special fiber $V(\pi)$ leads to the formal $o_L$-scheme $\fX=\Spf(A_{o_L,\pi})$; see \cite[I$_{\rm new}$, I.10.8.3]{EGA}. Its associated rigid analytic space $\fX_\rig$ (\cite{SFB378}, \cite{FvdP}) is given by the affinoid $L$-space $\fA(1):=\Spm(A_{o_L,\pi}[\piinv])$. This space can be regarded as the unit disc of the rigid analytic space $\fA(\infty)$ as it corresponds to ``radius of convergence $1$'', hence the notation.

We study the following instance of rigid analytic $\tau$-sheaves over $A_{o_L,\pi}[\piinv]$, in the sense of \cite{BH}.

\begin{Definition}
An \emph{analytic Anderson $A(1)$-motive} over $L$ is an object $\uM=(M,F_M)\in\FMod(A_{o_L,\pi}[\piinv])$ such that $\coker(F_M)$ is a finite-dimensional $L$-vector space and is annihilated by a power of $\fJ$. A \emph{morphism} of analytic Anderson $A(1)$-motives is defined as a morphism in the category $\FMod(A_{o_L,\pi}[\piinv])$. 
\end{Definition}

Here the prefix ``$A(1)$-'' indicates that we are considering an ana\-lytic variant of Anderson $A$-motives over the rigid analytic ``unit disc'' $\fA(1)$ in $\Spec(A_L)$.

\begin{Proposition}\label{ExpAlgTMot}
The natural functor $\FMod(A_L)\pfeil \FMod(A_{o_L,\pi}[\piinv])\,,\,\uM\gehtauf\uM\otimes_{A_L}A_{o_L,\pi}[\piinv]$ restricts to a functor $(\text{Anderson }A\text{-motives over }L)\pfeil (\text{analytic Anderson }A(1)\text{-motives over }L)$.\hfill\qed
\end{Proposition}

\begin{Definition}\label{DefModAndMotiv}
(a) Let $\uM_L\in\FMod(A_L)$ be an $F$-module over $A_L$. A \emph{model} of $\uM_L$ is a pair $(\ucM,\alpha)$ consisting of an object $\ucM\in\FMod(A_{o_L})$ and an isomorphism $\alpha\:\uM_L\isoto\ucM\otimes_{A_{o_L}}A_L$ inside $\FMod(A_L)$.

\smallskip\noindent
(b) Let $\uM_L\in\FMod(A_{o_L,\pi}[\piinv])$ be an $F$-module over $A_{o_L,\pi}[\piinv]$. A \emph{(formal) model} of $\uM_L$ is a pair $(\ucM,\alpha)$ consisting of an object $\ucM\in\FMod(A_{o_L,\pi})$ and an isomorphism $\alpha\:\uM_L\isoto \ucM\otimes_{A_{o_L,\pi}}A_{o_L,\pi}[\piinv]$ inside $\FMod(A_{o_L,\pi}[\piinv])$.

\smallskip\noindent
(c) In both cases a \emph{morphism} of models $\beta\:(\ucM,\alpha)\to(\ucM',\alpha')$ is an isomorphism $\beta\:\ucM\isoto\ucM'$ of $F$-modules satisfying $\alpha'=\beta[\piinv]\circ\alpha$. In particular the sets $\Hom\bigl((\ucM,\alpha),(\ucM',\alpha')\bigr)$ contain at most one element.\\[1mm]
We will sometimes drop the $\alpha$ from the notation and simply speak of $\ucM$ as a model of $\uM_L$.
\end{Definition}

For every $\ucM\in\FMod(A_{o_L})$, resp.\ $\ucM\in\FMod(A_{o_L,\pi})$ we can consider the reduction $\ucM\otimes_{A_{o_L}}A_\ell$, resp.\ $\ucM\otimes_{A_{o_L,\pi}}A_\ell$. Note, however, that this does \emph{not} induce a functor from $\FMod(A_{o_L})$, resp.\ $\FMod(A_{o_L,\pi})$ to $\FMod(A_\ell)$, since the induced $F$-map need not be injective. This circumstance lies at the origin of our study of good models:

\begin{Definition}\label{Def-goodFMod}
Let $\ucM$ be a model of an $F$-module $\uM_L$ over $A_L$, resp.\ over $A_{o_L,\pi}[\piinv]$. Then $\ucM$ is called a \emph{good} model if $\ucM/\pi\ucM$ is an $F$-module over $A_\ell$, i.e.\ if the induced $A_\ell$-linear map $$\bar{\sig}^*(\cM/\pi\cM)=(\cM/\pi\cM)\otimes_{A_\ell,\bar{\sig}}A_\ell\pfeil\cM/\pi\cM$$ is injective. 
\end{Definition}

If $\uM_L$ is an (analytic) Anderson motive there is an alternative notion of good reduction as follows.

\begin{Definition}\label{Def-good}
Let $\ucM$ be a model of an Anderson $A$-motive $\uM_L$, resp.\ of an analytic Anderson $A(1)$-motive $\uM_L$. Then $\ucM$ is called a \emph{good model in the strong sense} if $\coker(F_\cM)$ is a finite free $o_L$-module and is annihilated by $\fJ^d$, for some $d\geq 0$.
In this case we also say that $\ucM$ has \emph{good reduction over $o_L$}. 
\end{Definition}

\begin{Theorem}\label{ThmGoodModels}
Let $\ucM$ be a model of an Anderson $A$-motive, resp.\ of an analytic Anderson $A(1)$-motive $\uM_L$. Then $\ucM$ is a good model in the weak sense of Definition~\ref{Def-goodFMod} if and only if it is a good model in the strong sense of Definition~\ref{Def-good}.
\end{Theorem}

\begin{proof}
Since $\sig^*\cM$ is locally free over $A_{o_L}$, resp.\ over $A_{o_L,\pi}$, the natural map $\sig^*\cM\pfeil\sig^*M_L$ is injective and hence $F_\cM$ is injective because $F_{M_L}$ is. We obtain a short eqact sequence
\begin{equation}\label{EqSequenceModel}
0\longrightarrow\sig^*\cM\xrightarrow{\;F_\cM\,}\cM\longrightarrow\coker(F_\cM)\longrightarrow0\,.
\end{equation}
Let $\ucM$ be a good model in the strong sense. Tensoring the short exact sequence \eqref{EqSequenceModel} with $\ell$ over $o_L$ and using that $\coker(F_\cM)$ is supposed to be free over $o_L$ shows that the induced $A_\ell$-linear map $\bar{\sig}^*(\cM/\pi\cM)\pfeil\cM/\pi\cM$ remains injective. So $\ucM$ is a good model in the weak sense.

Conversely suppose that $\ucM$ is a good model in the weak sense. This time tensoring \eqref{EqSequenceModel} with $\ell$ over $o_L$ yields 
\[
0\longrightarrow {\rm Tor}_1^{o_L}\bigl(\coker F_\cM,\ell)\longrightarrow\sig^*\cM\otimes_{o_L}\ell\xrightarrow{\;F_\cM\otimes\id_\ell\,}\cM\otimes_{o_L}\ell\longrightarrow\coker(F_\cM)\otimes_{o_L}\ell\longrightarrow0\,.
\]
By assumption $F_\cM\otimes\id_\ell$ is injective, and so $0={\rm Tor}_1^{o_L}\bigl(\coker F_\cM,\ell)=\{x\in\coker(F_\cM)\:\pi x=0\}$ and $\coker(F_\cM)$ is flat over $o_L$ by \cite[Corollary~6.3]{EisCA}. This implies $\coker(F_\cM)\hookrightarrow\coker(F_\cM)\otimes_{o_L}L=\coker(F_{M_L})$. Since $\coker(F_{M_L})$ is annihilated by $\fJ^d$ for some $d$, the same is true for $\coker(F_\cM)$ which therefore is a finitely generated $A_{o_L}/\fJ^d$-module, resp.\ $A_{o_L,\pi}/\fJ^d$-module, and a fortiori a finitely generated $o_L$-module. Being flat, $\coker(F_\cM)$ is a finite free $o_L$-module. Thus $\cM$ is a good model in the strong sense.
\end{proof}

\begin{Remark}\label{RemGardeynGoodModel}
In \cite{GardeynSST}, Gardeyn develops a theory of semi-stable reduction of analytic Anderson $A(1)$-motives $\uM_L$. He shows that after replacing $L$ by a finite separable extension, $\uM_L$ has a model $\ucM$ such that the reduction $F_\cM\otimes\id_\ell$ is not nilpotent \cite[Proposition~3.3]{GardeynSST}. If $\,\,\overline{\!\!\ucM}'\subset\ucM/\pi\ucM$ is the maximal Frobenius $A_\ell$-submodule with injective $F_{\,\,\overline{\!\!\cM}'}$, he further shows that the support of $\coker(F_{\,\,\overline{\!\!\cM}'})$ is a finite set $S\subset \Spec A_\ell$. After removing $S$ from $\fA(1):=\Spm(A_{o_L,\pi}[\piinv])$ one can lift $\,\,\overline{\!\!\ucM}'$ to an $F$-submodule $\ucM'\subset\ucM|_{\fA(1)\smallsetminus S}$ which has good reduction in the weak sense of Definition~\ref{Def-goodFMod}; see \cite[Theorem~4.7]{GardeynSST}. As one sees from the following example, it is false in general that $S$ is the zero locus of $\fJ$ in $\Spec A_\ell$ and so we cannot expect that $\ucM'$ has good reduction in the strong sense of Definition~\ref{Def-good}.

Let $A=\IF[z]$ and $\zeta=c^*(z)\in\idm_L$. Then $\fJ=(z-\zeta)$. Let $\cM=o_L\langle z\rangle^{\oplus2}$ and $F_\cM={\,0 \enspace \pi(z-\zeta) \choose \,\pi \quad \,z-1\;\;}$. Then $\ucM=(\cM,F_\cM)$ is a model of the analytic Anderson $A(1)$-motive $\uM_L:=\ucM\otimes_{o_L}L$. The reduction $\ucM/\pi\ucM=\bigl(\ell[z]^{\oplus2},{\,0 \enspace\, 0\enspace\, \choose \,0 \;z-1}\bigr)$ contains the maximal Frobenius $A_\ell$-submodule $\,\,\overline{\!\!\ucM}'=\ell[z]\cdot{0\choose 1}$, whose Frobenius is $F_{\,\,\overline{\!\!\cM}'}=z-1$. So $S=V(z-1)\ne V(z)=V(\fJ)$.
\end{Remark}

\begin{Proposition}\label{ExpAlgGoodRed}
If $\uM_L$ is an Anderson $A$-Motive over $L$ having a (good) model $\ucM$ then its analytification $\uM_L\otimes_{A_L}A_{o_L,\pi}[\piinv]$ is an analytic Anderson $A(1)$-motive having the (good) model $\underline{\widehat{\cM}}:=\ucM\otimes_{A_{o_L}}A_{o_L,\pi}$ and the reduction $\underline{\widehat{\cM}}/\pi\underline{\widehat{\cM}}$ of $\underline{\widehat{\cM}}$ is canonically isomorphic to the reduction $\ucM/\pi \ucM$ of $\ucM$.
\end{Proposition}

\begin{proof}
The statement without the properties of being a good model is obvious. From the isomorphism $\underline{\widehat{\cM}}/\pi\underline{\widehat{\cM}}\isoto\ucM/\pi \ucM$ it follows that $\ucM$ is a good model in the sense of Definition~\ref{Def-goodFMod} if and only if $\underline{\widehat\cM}$ is a good model in the sense of Definition~\ref{Def-goodFMod}.
\end{proof}

Let us also mention the following result of Gardeyn on good reduction of Drinfeld $A$-modules.

\begin{Proposition}\label{PropGoodRedDriMod}
Let $\phi\:A\pfeil L[\tau]$ be a Drinfeld $A$-module over $L$; see \cite{DrinfeldEM} or \cite{Matzat}. Let $\uM=\uM(\phi)$ be the associated Anderson $A$-motive; see \cite[\S4.1]{Anderson} or \cite[\S8.1]{Gardeyn2}. Then the following are equivalent:
\begin{enumerate}
\item \label{PropGoodRedDriMod_A}
$\phi$ has good reduction over $o_L$, i.e.\ $\phi$ is isomorphic over $L$ to a Drinfeld $A$-module $\psi\:A\pfeil L[\tau]$ satisfying $\psi(A)\subset o_L[\tau]$ such that the reduction $\overline\psi\:A\pfeil o_L[\tau]\twoheadrightarrow\ell[\tau]$ is a Drinfeld $A$-module over $\ell$ of the same rank as $\psi$ and $\phi$.
\item \label{PropGoodRedDriMod_B}
$\uM$ has good reduction over $o_L$ in the weak and strong senses of Definitions~\ref{Def-good} and \ref{Def-goodFMod}.
\end{enumerate}
\end{Proposition}

\begin{proof}
Gardeyn~\cite[Theorem~8.1]{Gardeyn2} proved that $\phi$ has good reduction over $o_L$ if and only if $\uM$ has a good model in the weak sense. So the proposition follows from Theorem~\ref{ThmGoodModels}.
\end{proof}

\section{Local shtukas and analytic Anderson motives}\label{LocSht}

Anderson $A$-motives can be viewed as function-field analogs of abelian varieties. Barsotti-Tate groups, which can be associated with abelian varieties over $\IZ_p$-schemes, have effective local shtukas as function-field analogs.

\begin{Definition}\label{Def2.10}
An (\emph{effective}) \emph{local shtuka at $\eps$ over $o_L$} is an object $\uMhut=(\Mhut,F_{\Mhut})\in\FMod(A_{o_L,\epspi})$ such that $\coker(F_{\Mhut})$ is a finite free $o_L$-module and is annihilated by a power of $\fJ$. 
\end{Definition}

\begin{Remark}
If the residue field $\IF_\eps=A/\eps$ of $\eps$ is larger than $\IF$, i.e.\ if the degree $d_\eps:=[\IF_\eps:\IF]>1$, the ring $A_{o_L,\epspi}$ is not an integral domain but a product $\displaystyle A_{o_L,\epspi}=\prod_{i\in\IZ/d_\eps\IZ}A_{o_L,\epspi}/\fa_i$ of integral domains. To describe this product decomposition, note that $A_{o_L,\epspi}=\pl_n\, A_{o_L}/\eps^n=\pl_n\,(A/\eps^n)\otimes_\IF o_L = A_\eps\widehat\otimes_\IF o_L$. By Cohen's structure theorem $A_\eps\cong\IF_\eps\lb z_\eps\rb$ for a uniformizer $z_\eps$ of $A$ at $\eps$. Then $\fa_i=(\alpha\otimes1-1\otimes c^*(\alpha)^{\r^i}\:\alpha\in\IF_\eps\subset A_\eps)$, where we use that $c^*\:A\pfeil o_L$ factors through $c^*\:A_\eps\pfeil o_L$. The factors $A_{o_L,\epspi}/\fa_i$ are isomorphic to $o_L\lb z_\eps\rb$ and hence are integral domains. They are cyclically permuted by $\sig$ because $\sig(\fa_i)=\fa_{i+1}$. By \cite[Proposition~8.8]{BoHa} the functor $(\Mhut,F_{\Mhut})\gehtauf (\Mhut/\fa_0\Mhut, (F_{\Mhut})^{d_\eps})$ is an equivalence between the category of effective local shtukas at $\eps$ over $o_L$ as in Definition~\ref{Def2.10} and the category of pairs $(\hat M_0,\widetilde F_{\hat M})$ where $\hat M_0$ is a free module of finite rank over $A_{o_L,\epspi}/\fa_0$ and $\widetilde F_{\hat M}\:(\sig^{d_\eps})^*\hat M_0\pfeil\hat M_0$ is injective with $\coker(\widetilde F_{\hat M})$ being a finite free $o_L$-module. In \cite{HDict,HHabil} these pairs $(\Mhut_0,\widetilde F_{\Mhut})$ are called (\emph{effective}) \emph{local shtukas}.
\end{Remark}

The following criterion for good reduction of analytic Anderson $A(1)$-motives can be regarded as a \emph{good-reduction Local-Global Principle at the characteristic place}.

\begin{Theorem}\label{GoodRedCrit}
Let $\uM_L=(M_L,F_{M_L})$ be an analytic Anderson $A(1)$-motive over $L$ such that $\coker(F_{M_L})$ is annihilated by $\fJ^d$ for some $d$. Then the following assertions are equivalent:
\begin{enumerate}
\item $\uM_L$ admits a good model in the strong sense of Definition~\ref{Def-good}.
\item There is an effective local shtuka $\uMhut=(\Mhut, F_{\Mhut})$ at $\eps$ over $o_L$ such that $\coker(F_{\Mhut})$ is annihilated by $\fJ^d$, and an isomorphism $\uM_L\otimes_{A_{o_L,\pi}[\piinv]}A_{o_L,\epspi}[\piinv]\cong \uMhut\otimes_{A_{o_L,\epspi}}A_{o_L,\epspi}[\piinv]$ in $\FMod(A_{o_L,\epspi}[\piinv])$.
\end{enumerate}
\end{Theorem}

\begin{proof}
1. In order to show that (ii) implies (i), let $f\: M_L\otimes\forget{_{A_{o_L,\pi}[\piinv]}}A_{o_L,\epspi}[\piinv]\isoto\Mhut\otimes\forget{_{A_{o_L,\epspi}}}A_{o_L,\epspi}[\piinv]=:\Mhut[\piinv]$ be an $F$-equivariant isomorphism of $A_{o_L,\epspi}[\piinv]$-modules as in (ii). We have canonical $F$-equivariant $A_{o_L,\pi}$-linear maps 
$$i\: M_L \pfeil M_L\otimes_{A_{o_L,\pi}[\piinv]}A_{o_L,\epspi}[\piinv],\qquad j\: \Mhut \pfeil \Mhut[\piinv]$$
where $i$ (resp., $j$) is injective since $M_L$ (resp., $\Mhut$) is flat. Consider the $A_{o_L,\pi}$-module $\cM=\im(i)\cap f^\inv(\im(j)).$
We will show that $\cM$ is a good model of $\uM_L$. 
The inclusion $\cM\haken M_L$ gives rise to an $A_{o_L,\pi}[\piinv]$-linear embedding $\cM[\piinv]\haken M_L[\piinv]\cong  M_L$, which is in fact an isomorphism, because if $m\in M_L$ there is an $s\geq0$ such that $\pi^s f(m\otimes 1)\in \im(j)$, i.e.\ $\pi^sm\in\cM$.

\smallskip\noindent
2. In order to show that $\cM$ is a finitely generated $A_{o_L,\pi}$-module we use the embedding $\iota\:\IF[z]\pfeil A$ from Remark~\ref{t} and the induced maps $L\langle z\rangle\pfeil A_{o_L,\pi}[\piinv]$ and $o_L\lb z\rb\pfeil A_{o_L,\epspi}$ from \eqref{CProj-Abb}. 
Let $(e_1,...,e_m)$ be a basis of $M_L$ over the principal ideal domain $L\la z\ra$. Furthermore, let $(d_1,...,d_n)$ be a basis for $\Mhut$ over the local ring $o_L\lb z\rb$. Note that the basis $(e_1,...,e_m)$ gives rise to an isomorphism $M_L\otimes_{L\la z\ra}o_L\lb z\rb[\piinv]\cong o_L\lb z\rb[\piinv]^{\oplus m}$. For every $\nu=1,...,n$ we consider $ f^\inv(d_\nu)$ and regard it as an element of the right-hand side of this isomorphism. We choose $N\geq 0$ big enough, such that $ f^\inv(\pi^N d_\nu)\in o_L\lb z\rb^{\oplus m}$ for all $\nu$, say 
$$f^\inv(\pi^N d_\nu)\;=\; (\rho_{\nu,1},...,\rho_{\nu,m})$$
where $\rho_{\nu,\mu}\in o_L\lb z\rb$. Now let $x\in\cM$. Via $ f$ we obtain $ f(x)=\sum_\nu\lambda_\nu d_\nu$ in $\Mhut$, with suitable $\lambda_\nu\in o_L\lb z\rb$. Consequently $ f(\pi^N x)=\sum_\nu \lambda_\nu(\pi^Nd_\nu)$, so that the image of $\pi^Nx$ in $o_L\lb z\rb^{\oplus m}$ satisfies $\pi^Nx=\sum_\mu(\sum_\nu\lambda_\nu\rho_{\nu,\mu})e_\mu$. The appearing scalars $h_\mu=\sum_\nu\lambda_\nu\rho_{\nu,\mu}$ have, in fact, to be elements of $L\la z\ra\cap o_L\lb z\rb=o_L\la z\ra$. Inside $M_L$ we may write $x=\pi^{-N}\pi^Nx=\sum_\mu h_\mu\pi^{-N}e_\mu$, so that we may conclude 
$$\cM\subset\sum_\mu o_L\la z\ra \pi^{-N}e_\mu.$$
Being a submodule of a finitely generated module over a noetherian ring, $\cM$ has to be a finitely generated $o_L\la z\ra$-module and hence a finitely generated $A_{o_L,\pi}$-module.

\smallskip\noindent
3. We claim that $\cM/\pi\cM$ is torsion-free and hence free over $\ell[z]$, because it is finitely generated. 
Let $x\in\cM$, and let $\lambda\in o_L\la z\ra$ be such that $\lambda\notin\pi o_L\la z\ra$ and $\lambda x\in \pi\cM$, say $\lambda x=\pi y$ for some $ y\in\cM$. In order to prove that $\cM/\pi \cM$ is torsion-free we must show that $x\in \pi\cM$. First suppose that $\lambda\in o_L\la z\ra\cap o_L\lb z\rb^\times$. We consider $\pi^\inv x\in M_L$. In fact, this element lies in $\cM$, since we have $ f(\pi^\inv x)=\lambda^\inv f( y)\in\Mhut$. Consequently $x=\pi(\pi^\inv x)\in\pi\cM$. 

Let us next assume that $\lambda=z^n$ and show that $z^nx\in\pi\cM$ implies $x\in\pi\cM$ for any $n\geq 0$. By induction, it suffices to consider the case $n=1$. So suppose $zx\in\pi\cM$, say $zx=\pi y$. Let $ f(x)=\sum_\nu \beta_\nu d_\nu$, where $(d_1,...,d_n)$ is the finite $o_L\lb z\rb$-basis of $\Mhut$ fixed before. The relation $zx=\pi y$ implies that $\pi \mid z\beta_\nu$ for every index $\nu$, so that $\pi\mid\beta_\nu$ for every $\nu$. Therefore $\pi^\inv x\in M_L$ necessarily maps via $ f$ to an element of $\Mhut$, i.e.\ $x\in \pi\cM$. 

Finally we treat the case for general $\lambda=\sum_s\lambda_s z^s$ and suppose that $\lambda\notin o_L\lb z\rb^\times$, that is $\pi\mid \lambda_0$. This means we find $\lambda'\in o_L[z]$ and $\lambda''\in o_L\la z\ra\cap o_L\lb z\rb^\times$ such that $\lambda=\pi\lambda'+z^N\lambda''$ for some $N\geq 1$. We have $\pi y=\lambda x=\pi\lambda'x+z^N\lambda''x$. In particular $z^N\lambda''x=\pi(y-\lambda'x)\in\pi\cM$ and by the above $\lambda''x\in\pi\cM$ and $x\in\pi\cM$.

Thus we have proved that $\cM/\pi\cM$ is free over $\ell[z]$. It follows that $\cM/\pi\cM$ is locally free of finite rank over $A_\ell$.

\smallskip\noindent
4. We claim that $\cM$ is locally free of finite rank over $A_{o_L,\pi}$. Since it is finitely generated it only remains to show that $\cM$ is flat over $A_{o_L,\pi}$. 
Since $A_{o_L,\pi}$ is $\pi$-adically complete and separated, $\pi A_{o_L,\pi}$ is contained in the Jacobson radical $\jac(A_{o_L,\pi})$ by \cite[Theorem~8.2]{Matsu}, and the $A_{o_L,\pi}$-module $\cM$ is finitely generated, so that $\cM$ is $\pi$-adically \emph{ideally Hausdorff} in the sense of \cite[III.5.1]{BourbakiCA}. In the preceding step we have shown that $\cM/\pi\cM$ is flat over $A_\ell\cong  A_{o_L,\pi}/\pi A_{o_L,\pi}$, and we know that $\cM$ has no $\pi$-torsion, so that the canonical map $\pi A_{o_L,\pi}\otimes_{A_{o_L,\pi}}\cM\pfeil \pi\cM$ is an isomorphism. Therefore, by Bourbaki's Flatness Criterion \cite[\S\,III.5.2, Th\'eor\`eme~1(iii)]{BourbakiCA}, we may conclude that $\cM$ is indeed flat over $A_{o_L,\pi}$. 

\smallskip\noindent
5. We note that $\sig^*\cM=\sig^*\im(i)\cap (\sig^*f)^\inv(\sig^*\im(j))$ because the functor $\sig^*$ is exact by Lemma~\ref{Frob-cocart}. By the $F$-equivariance of $f$ we obtain a Frobenius $F_\cM\:\sig^*\cM\pfeil\cM$. It is injective because $F_{M_L}$ is. We set $\ucM:=(\cM,F_\cM)$.

\smallskip\noindent
6. Next we claim that $\fJ^d\coker(F_\cM)=0$. Let $x=\sum_\nu h_\nu m_\nu\in \fJ^d\cM$ where $h_\nu\in\fJ^d$ and $m_\nu\in\cM$. Since $\coker(F_{M_L})$ is annihilated by $\fJ^d$, there is a (unique) $y\in \sig^*M_L$ such that $x=\sum_\nu h_\nu m_\nu=F_{M_L}(y)$. We have to show that $y\in\sig^*\cM=\sig^*\im(i)\cap (\sig^*f)^\inv(\sig^*\im(j))$. So it remains to see that $(\sig^*f)(y)\in\im(\sig^*j)$. Indeed, inside $\Mhut[\piinv]$ we have $f(x)=f(F_{M_L}(y))=F_{\Mhut}((\sig^*f)(y))$. On the other hand, the linearity of $f$ and $j$ gives that $f(x)=\sum_\nu h_\nu f(m_\nu\otimes 1)=j(y')$ for some $y'\in \fJ^d \Mhut\subset \im(F_{\Mhut})$, say $y'=F_{\Mhut}(y'')$ for a $y''\in \sig^*\Mhut$. Thus $f(x)=F_{\Mhut}((\sig^*j)(y''))$. So finally, since $F_{\Mhut}\:\sig^*\Mhut[\piinv]\pfeil \Mhut[\piinv]$ is injective, we obtain that $(\sig^*f)(y)=(\sig^*j)(y'')$, as desired.

\smallskip\noindent
7. Finally we show that the kernel $V$ of $\oF\:\sig^*(\cM/\pi\cM)\pfeil\cM/\pi\cM$ is trivial. This implies that $\ucM$ is a good model of $\uM_L$ in the weak sense of Definition~\ref{Def-goodFMod}, which is enough by Theorem~\ref{ThmGoodModels}. 

We have already shown that $\fJ^d\cM\subset\im(F_\cM)$. Since $(z-\zeta)\in\fJ$ for $\zeta:=c^*(z)\in o_L$ we have a chain of $o_L\la z\ra$-modules $(z-\zeta)^d\cM\subset\im(F_\cM)\subset\cM$. The element $\zeta\in o_L$ is zero mod $\pi$, and we obtain 
\begin{equation}\label{EqImoF}
z^d(\cM/\pi\cM)\subset\im(\oF)\subset \cM/\pi\cM.
\end{equation}
We know that $\cM/\pi\cM$ is finite free over $\ell[z]$. Therefore the middle term $W:=\im(\oF)$ in the latter chain has full rank inside $\cM/\pi\cM$. Finally, taking ranks in the (split) short exact sequence of finite free $\ell[z]$-modules
$$0\pfeil V\pfeil \sig^*(\cM/\pi\cM) \xrightarrow{\;\oF\,} W\pfeil 0$$
accomplishes the proof that $V$ indeed is trivial.
\forget{

\smallskip\noindent
8. It remains to prove that the co\-kernel $C$ of $F_\cM\:\sig^*\cM\pfeil\cM$ is a finite free $o_L$-module. 
In a first step we show that $C$ is finitely generated over $o_L$. In equation~\eqref{EqImoF} we saw that $C/\pi C$ is a quotient of $(\cM/\pi\cM)\big/z^d(\cM/\pi\cM)$ and hence a finite dimensional $\ell$-vector space. Since $\pi\in\jac(A_{o_L,\pi})$ and since $C$, being a quotient of $\cM$, is finitely generated over $A_{o_L,\pi}$, we conclude by Krull's Theorem \cite[Theorem~8.10]{Matsu} that $C$ is $\pi$-adically separated. By \cite[Theorem~8.4]{Matsu} it follows that $C$ is finitely generated.

In a second step we show that $C$ is a flat $o_L$-module, which will imply that $C$ is finite free over the local ring $o_L$. Since we have just seen that $C/\pi C$ is free and hence flat over $\ell$, we only need to prove that $C$ has trivial $\pi$-torsion. Then Bourbaki's Flatness Criterion \cite[\S\,III.5.2, Th\'eor\`eme~1(iii)]{BourbakiCA}, will yield the desired result.
So let $x\in\cM$ with $\pi x=F_\cM(y)\in \im(F_\cM)$ for an element $y\in \sig^*\cM$. 
Denoting residues modulo $\pi\cM$ by a bar, we see that $0=\overline{\pi x}=\oF(\bar y)$. By the injectivity of $\oF$ we must have $y=\pi y'$ for a $y'\in\sig^*\cM$ and $x=F_\cM(y')\in\im(F_\cM)$. Thus $C$ is finite free over $o_L$ and we have shown that $\ucM$ is a good model for $\uM_L$.
}

\bigskip\noindent
8. Conversely, in order to show that (i) implies (ii), suppose that $(\ucM,\alpha)$ is a good model of $\uM_L$. We define 
$$\uMhut=\ucM\otimes_{A_{o_L,\pi}}A_{o_L,\epspi},$$ 
i.e.\ $\uMhut$ equals the completion of $\ucM$ for the $\epspi A_{o_L,\pi}$-adic topology. It is clear that the $F$-equivariant isomorphism $\alpha\:M_L\isoto\cM[\piinv]$ of $A_{o_L,\pi}[\piinv]$-modules gives rise to a natural $F$-equivariant $A_{o_L,\epspi}[\piinv]$-linear isomorphism
$M_L\otimes_{A_{o_L,\pi}[\piinv]}A_{o_L,\epspi}[\piinv]\cong  \Mhut[\piinv].$

We claim that $\uMhut$ is a local shtuka. Indeed, by base change, $\Mhut$ is again locally free of finite rank. Furthermore, since the completion map $A_{o_L,\pi}\pfeil A_{o_L,\epspi}$ is Frobenius-equivariant and flat, we obtain an injective map $\Mhut\otimes_{(A_{o_L,\epspi}),\sig}A_{o_L,\epspi}\pfeil \Mhut.$
Let $C'$ be its cokernel, and let $C=\coker(F_\cM)$, i.e.\ $C'\cong  C\otimes_{A_{o_L,\pi}}A_{o_L,\epspi}.$
Since $C$ is annihilated by $\fJ^d$ the module $C'$ equals $C$ and it is finite free over $o_L$. Thus $\uMhut$ is an effective local shtuka over $o_L$.
\end{proof}

\begin{Remark}\label{RemNotEquiv}
Steps 1-4 in the previous proof suggest that there is an equivalence of categories
\begin{eqnarray*}
\cF\:\left\{\begin{array}{l} \text{finite locally free} \\ \text{$A_{o_L,\pi}$-modules $\cM$}\end{array}\right\} & \stackrel{\sim}{\longleftrightarrow} &
\left\{\begin{array}{l}
\text{triples }(M_L,\Mhut,f)\text{ consisting of}\\
\text{\textbullet\ a finite locally free }A_{o_L,\pi}[\piinv]\text{-module }M_L,\\
\text{\textbullet\ a finite locally free }A_{o_L,\epspi}\text{-module }\Mhut,\text{ and}\\
\text{\textbullet\ an isomorphism of }A_{o_L,\epspi}[\piinv]\text{-modules}\\
\quad f\: M_L\otimes_{A_{o_L,\pi}[\piinv]}A_{o_L,\epspi}[\piinv]\isoto \Mhut\otimes_{A_{o_L,\epspi}}A_{o_L,\epspi}[\piinv]\end{array}\right\} \\[2mm]
\cM & \longmapsto & \bigl(\cM\otimes_{A_{o_L,\pi}}A_{o_L,\pi}[\piinv],\,\cM\otimes_{A_{o_L,\pi}}A_{o_L,\epspi},\,\id_{\cM\otimes A_{o_L,\epspi}[\piinv]}\bigr)\,,
\end{eqnarray*}
where on the right a morphism $\underline h=(h_L,\hat h)\:(M_L,\Mhut,f)\pfeil(M_L',\Mhut',f')$ consists of a morphism $h_L\:M_L\to M_L'$ and a morphism $\hat h\:\Mhut\pfeil\Mhut'$ such that $f'\circ (h_L\otimes\id_{A_{o_L,\epspi}[\piinv]})=(\hat h\otimes\id_{A_{o_L,\epspi}[\piinv]})\circ f$.

However, \emph{this is false} as can be seen from the following example, where we take $A=\IF[z]$. We choose an element $a\in\ell\lb z\rb\subset\ell\lbc z\rbc$ such that $a\notin \ell(z)$, and we let $\Delta= \left(\begin{smallmatrix} 1 & \,\pi^{-1}a \\ 0 & \,\pi^{-1} \end{smallmatrix}\right)$. Set $M_L=L\langle z\rangle^{\oplus2},\,\Mhut=\Delta\cdot o_L\lb z\rb^{\oplus2}$ and $f=\id_{o_L\lb z\rb[\piinv]^2}$. Then $\Delta^{-1}= \left(\begin{smallmatrix} 1 & -a \\ 0 & \pi \end{smallmatrix}\right)\in o_L\lb z\rb^{2\times2}$ and
\[
o_L\lb z\rb^{\oplus2}\;=\;\Delta\cdot\Delta^{-1}o_L\lb z\rb^{\oplus2}\;\subset\; \Mhut \;\subset\; \pi^{-1}o_L\lb z\rb^{\oplus2}\,.
\]
If there was a finite free $A_{o_L,\pi}$-module $\cM$ with $(h_L,\hat h)\:\cF(\cM)\isoto(M_L,\Mhut,f)$, then it had to satisfy $\cM\cong M_L\cap\Mhut$ with $h_L$ and $\hat h$ induced from the inclusions $M_L\cap\Mhut\subset M_L$ and $M_L\cap\Mhut\subset\Mhut$. So we may take directly $\cM:= M_L\cap\Mhut$. It satisfies $o_L\langle z\rangle^{\oplus2}\subset\cM\subset \pi^{-1}o_L\langle z\rangle^{\oplus2}$. We claim that, in fact, the first inclusion is an equality. Namely let $\left(\begin{smallmatrix} v \\ w \end{smallmatrix}\right)=\left(\begin{smallmatrix} \pi^{-1}v_0+v' \\ \pi^{-1}w_0+w' \end{smallmatrix}\right)\in\cM$ with $v_0,w_0\in\ell[z]$ and $v',w'\in o_L\langle z\rangle$. Then $\Delta^{-1}\left(\begin{smallmatrix} v \\ w \end{smallmatrix}\right)=\left(\begin{smallmatrix} \pi^{-1}v_0+v'-\pi^{-1}aw_0-aw' \\ w_0+\pi w' \end{smallmatrix}\right)\in o_L\lb z\rb^{\oplus2}$. This implies $v_0=aw_0$ in $\ell\lb z\rb$. If $w_0\ne0$ we get $a=v_0/w_0\in\ell(z)$ in contradiction to our assumption. So $w_0=v_0=0$ and $\left(\begin{smallmatrix} v \\ w \end{smallmatrix}\right)\in o_L\langle z\rangle^{\oplus2}$. This proves our claim that $\cM=o_L\langle z\rangle^{\oplus2}$. We conclude that $\cF(\cM)\not\cong(M_L,\Mhut,f)$ and $\cF$ is not an equivalence of categories.
\end{Remark}

After this example the following result is even more surprising.

\begin{Corollary}\label{GoodModel-LocalShtuka}
Let $\uM_L$ be an analytic Anderson $A(1)$-motive over $L$. Then there is an equivalence of categories
\begin{eqnarray*}
\left\{\begin{array}{l} \text{good models $(\ucM,\alpha)$ of }\uM_L\text{ in the}\\ \text{sense of Definitions~\ref{Def-good} and \ref{Def-goodFMod}}\end{array}\right\} & \stackrel{\sim}{\longleftrightarrow} &
\left\{\begin{array}{l}
\text{pairs }(\uMhut,f)\text{ consisting of}\\
\text{\textbullet\ a local shtuka }\uMhut\text{ at }\eps\text{ over }o_L,\text{ and}\\
\text{\textbullet\ an isomorphism in }\FMod(A_{o_L,\epspi}[\piinv])\\
\quad f\: \uM_L\otimes A_{o_L,\epspi}[\piinv]\isoto \uMhut[\piinv]\end{array}\right\} \\[2mm]
(\ucM,\alpha) & \longmapsto & (\ucM,\alpha)\otimes_{A_{o_L,\pi}}A_{o_L,\epspi}\,,
\end{eqnarray*}
where on the right-hand side a \emph{morphism} of pairs $\hat\beta\:(\uMhut,f)\isoto(\uMhut',f')$ is defined to be an isomorphism of local shtukas $\hat\beta\:\uMhut\isoto\uMhut'$ satisfying $f'=\hat\beta\circ f$.
\end{Corollary}

\begin{proof}
Suppose that $(\ucM,\alpha)$ is a good model of $\uM_L$. In the proof of \ref{GoodRedCrit} we have seen that its completion $\ucMhut:=\ucM\otimes_{A_{o_L,\pi}}A_{o_L,\epspi}$ is a local shtuka at $\eps$. The $F$-equivariant isomorphism $\alpha\:M_L\isoto\cM[\piinv]$ of $A_{o_L,\pi}[\piinv]$-modules induces the isomorphism
$$f:=\alpha\otimes\id_{A_{o_L,\epspi}[\piinv]}\: M_L\otimes_{A_{o_L,\pi}[\piinv]}A_{o_L,\epspi}[\piinv]\isoto \hat{\cM}\otimes_{A_{o_L,\epspi}}A_{o_L,\epspi}[\piinv]$$
which is $F$-equivariant, and satisfies $\cM=f(M_L)\cap\hat{\cM}$, because $A_{o_L,\pi}=A_{o_L,\pi}[\piinv]\cap A_{o_L,\epspi}$. 

\medskip
To see that this functor is fully faithful let $(\ucM,\alpha)$ and $(\ucM',\alpha')$ be good models of $\uM_L$ and let $\hat\beta\:(\ucMhut,f):=(\ucM,\alpha)\otimes_{A_{o_L,\pi}}A_{o_L,\epspi}\isoto(\ucMhut{}',f'):=(\ucM',\alpha')\otimes_{A_{o_L,\pi}}A_{o_L,\epspi}$ be an isomorphism. This means $f'=\hat\beta\circ f$. Applying $\cM=f(M_L)\cap\hat\cM$ and $\cM'=f'(M_L)\cap\hat\cM'$ we see that $\hat\beta(\cM)=\cM'$. Therefore $\beta:=\hat\beta|_\cM:\cM\isoto\cM'$ is the desired isomorphism satisfying $\beta\otimes\id_{A_{o_L,\epspi}}=\hat\beta$. This implies $\alpha'=\beta\circ\alpha$ and the $F$-equivariance of $\beta$, and hence $\beta\:(\ucM,\alpha)\isoto(\ucM',\alpha')$.

\medskip
To prove essential surjectivity, let a local shtuka $\uMhut$ together with an isomorphism $f\:\uM_L\otimes_{A_{o_L,\pi}[\piinv]}A_{o_L,\epspi}[\piinv]\isoto \uMhut[\piinv]$ be given. It remains to show that the $\epspi A_{o_L,\pi}$-adic completion $\ucMhut:=\ucM\otimes_{A_{o_L,\pi}}A_{o_L,\epspi}$ of the good model $\cM=M_L\cap f^\inv(\Mhut)$ gained in the proof of \ref{GoodRedCrit} gives back $\uMhut$. Then we take $\alpha$ as the canonical isomorphism $\id\:\cM\otimes_{A_{o_L,\pi}}A_{o_L,\pi}[\piinv]\isoto M_L$. By construction of $\ucM$, the map $f$ restricts to an embedding $\cM\haken\Mhut$, which in turn induces an $F$-equivariant and $A_{o_L,\epspi}$-linear map $\psi:=f|_{\hat\cM}\: \hat\cM\pfeil \Mhut$, which becomes an isomorphism after inverting $\pi$.
Our aim is to show that already the map $\psi$ is an isomorphism $(\ucM,\id)\otimes_{A_{o_L,\pi}}A_{o_L,\epspi}\isoto(\uMhut,f)$. According to Remark~\ref{RemNotEquiv} we have to use the Frobenius morphisms $F_{\hat\cM}$ and $F_{\Mhut}$ in an essential way.

We know that $\cM$ is finite free over $o_L\la z\ra$ and that $\rk_{o_L\lb z\rb}(\hat\cM)=\rk_{o_L\lb z\rb}(\Mhut)=:s.$
We fix an $o_L\lb z\rb$-basis $\fB$ (resp., $\fC$) of $\hat\cM$ (resp., of $\Mhut$) and let $\bA=\boo_\fC[\psi]_\fB\in o_L\lb z\rb^{s\times s}$ be the matrix which describes $\psi$ with respect to $\fB$ and $\fC$. Likewise, we let $$\bT=\boo_\fB[F_{\hat\cM}]_{\sig^*\fB},\qquad \bT'=\boo_\fC[F_{\Mhut}]_{\sig^*\fC}$$  
be the matrices corresponding to $F_{\hat\cM}$ and $F_{\Mhut}$, so that $\bA\bT=\bT'\sig(\bA)$ by virtue of the $F$-equivariance of $\psi$. In order to see that $\psi$ is an isomorphism, we need to show that $\det(\bA)$ is a unit in $o_L\lb z\rb$. To begin with, an elementary application of the Weierstra{\ss} Division Theorem for $o_L\lb z\rb$ (\cite[VII.3.8.5]{BourbakiCA}) shows that the kernel of the epimorphism $o_L\lb z\rb\pfeil o_L$, $z\gehtauf\zeta$, is generated by $z-\zeta$, so that the latter is a prime element of $o_L\lb z\rb$. Furthermore, recall that $o_L\lb z\rb$, being a regular local ring, is factorial (\cite{Matsu}, 20.3). We know that $\ucMhut$ is a local shtuka, so that $F_{\hat\cM}$ becomes an isomorphism after inverting $z-\zeta$ which means that $\det(\bT)^\inv$ lies in $o_L\lb z\rb[\frac{1}{z-\zeta}]$. Say we have a relation $(z-\zeta)^e=\det(\bT)u$ in $o_L\lb z\rb $, for some $e\geq 0$ and some $u\in o_L\lb z\rb$. By a comparison of powers of $z-\zeta$, we may assume that $u$ is not divisible by $z-\zeta$. In this equation there is only one prime element of $o_L\lb z\rb$ occurring on both sides, which, by factoriality, implies that $u$ has to be a unit in $o_L\lb z\rb$. Let $(z-\zeta)^{e'}=\det(\bT')u'$ be the corresponding relation for the local shtuka $\Mhut$, with a unit $u'\in o_L\lb z\rb^\times$ and some suitable $e'\geq 0$. Since $\hat\cM\pfeil \Mhut$ becomes an isomorphism after inverting $\pi$, we see that $\det(\bA)\in o_L\lb z\rb[\piinv]^\times$. Note that the natural reduction-mod-$z$ map $o_L\lb z\rb\pfeil o_L$, $h\gehtauf h(0)$, induces an epimorphism of abelian groups $o_L\lb z\rb[\frac{1}{\pi}]^\times\pfeil L^\times$, so that the absolute term $\delta:=\det(\bA)(0)$ of $\det(\bA)$ lies in $L^\times$. By virtue of the relations derived above, the equation $\det(\bA)\det(\bT)=\det(\bT')\sig(\det(\bA))$ yields $$\det(\bA)u^\inv(z-\zeta)^e=u'^\inv(z-\zeta)^{e'}\sig(\det(\bA))$$ which modulo $z$ gives $\delta^{q-1} = \frac{u'(0)}{u(0)}(-\zeta)^{e-e'}$ in $L^\times$. Suppose for a moment that $e=e'$. In this case it follows at once that $\delta$ is a unit in $o_L$, so that $\det(\bA)$ is a unit in $o_L\lb z\rb$. Therefore it remains to verify that our assumption $e=e'$ is justified. This can be seen as follows: The reduction-mod-$\pi$ map $o_L\lb z\rb\pfeil \ell\lb z\rb$ is an epimorphism with kernel $\pi o_L\lb z\rb$, and via applying the functor $\cdot\otimes_{o_L\lb z\rb}\ell\lb z\rb$ to $F_{\Mhut}\: \sig^*\Mhut\pfeil \Mhut$ we obtain a commutative diagram
$$\xymatrix{\sig^*\Mhut=\Mhut\otimes_{o_L\lb z\rb,\sig} o_L\lb z\rb \ar[r] \ar[d] & \Mhut \ar[d]\\
\bar{\sig}^*\Mhut/\pi\Mhut=\Mhut/\pi\Mhut \otimes_{\ell\lb z\rb,\bar{\sig}}\ell\lb z\rb \ar[r] & \Mhut/\pi\Mhut}$$
where in the upper row (resp., the bottom row) both modules are finite free of the same rank over $o_L\lb z\rb$ (resp., over $\ell\lb z\rb$) and the arrow is given by $F_{\Mhut}$ (resp., by $\bar{F}=F_{\Mhut}\otimes \id_{\ell\lb z\rb}$). The reduced matrix $\overline{\bT'}\in \ell\lb z\rb^{s\times s}$ describes the map $\bar{F}$ with respect to the $\ell\lb z\rb$-bases $\overline{\sig^*\fC}=\bar{\sig}^*\bar{\fC}$ of $\bar{\sig}^*\Mhut/\pi\Mhut$ and $\bar{\fC}$ of $\Mhut/\pi\Mhut$ respectively, and from what we have seen before, we derive the relation $\det(\overline{\bT'})\overline{u'}=z^{e'}$, i.e.\ $e'=\ord_z(\det(\overline{\bT'})),$ the latter being true since $\overline{u'}\in\ell\lb z\rb^\times$. In particular we have $\det(\overline{\bT'})\in \ell\lb z\rb-\{0\}$. A similar observation for the local shtuka $\hat\cM$ instead of $\Mhut$ shows that $e=\ord_z(\det(\overline{\bT}))$. Let $C=\coker(F_{\hat\cM})$ and $C'=\coker(F_{\Mhut})$. 
Multiplication with the matrix $\overline{\bT'}$ gives rise to a finite presentation 
$\ell\lb z\rb^s\pfeil \ell\lb z\rb^s\pfeil C'/\pi C'\pfeil 0.$
Taking determinants in an equation of the form $\bS_1\overline{\bT'}\bS_2=\Diag(a_1,...,a_d,0,0,...,0)$, where $\bS_1, \bS_2\in\GL_s(\ell\lb z\rb)$ are suitable matrices such that $a_1,...,a_d\in \ell\lb z\rb-\{0\}$ are the elementary divisors of $\overline{\bT'}$ (see \cite{BourbakiA47}, VII.4.5.1), yields that necessarily $d=s$, so that $C'/\pi C'$ is a torsion $\ell\lb z\rb$-module and 
$$C'/\pi C' \cong \ell\lb z\rb/a_1\ell\lb z\rb\oplus ...\oplus \ell\lb z\rb/a_s\ell\lb z\rb \cong  \ell^{n_1}\oplus ...\oplus \ell^{n_s}$$
where $n_j=\ord_z(a_j)$ and $\sum_j n_j=e'$, i.e.\ $e'=\ord_z(\det(\overline{\bT'}))=\rk_\ell(C'/\pi C')=\rk_{o_L}(C'),$
the latter equation being valid since $C'/\pi C'\cong  C'\otimes_{o_L\lb z\rb}\ell\lb z\rb$. Finally, imitating this argument for the local shtuka $\hat\cM$ yields that $e=\ord_z(\det(\overline{\bT}))=\rk_\ell(C/\pi C)=\rk_{o_L}(C).$
So it remains to show that $\rk_{o_L}(C)=\rk_{o_L}(C')$. Indeed, we know that $\psi\: \hat\cM\pfeil \Mhut$ gives back $f$ in the generic fiber, which means that $\psi$ is an isomorphism after inverting $\pi$. Therefore, inverting $\pi$ in the commutative diagram with exact rows
$$\xymatrix{0 \ar[r] & \sig^*(\hat\cM) \ar[r] \ar[d]_{\sig^*\psi} & \hat\cM \ar[r] \ar[d]_\psi & C \ar[r] \ar[d] & 0\\
0 \ar[r] & \sig^*\Mhut \ar[r] & \Mhut \ar[r] & C' \ar[r] & 0}$$ 
exhibits $(\sig^*\psi)[\piinv]=\sig^*(\psi[\piinv])$ and $\psi[\piinv]$ as $o_L\lb z\rb[\piinv]$-linear isomorphisms, so that the Snake Lemma yields $C'[\piinv]\cong  C[\piinv]$, and we obtain
$\rk_{o_L}(C')=\dim_L(C'[\piinv])=\dim_L(C[\piinv])=\rk_{o_L}(C),$
as desired.
\end{proof}

\section{The reduction criterion for Anderson motives}\label{tauSh}

\begin{Definition}
(a) Let $\ucM\in\FMod(A_{o_L})$. Following Gardeyn \cite{GardeynSST}, $\ucM$ is called \emph{$A_{o_L}$-maximal} if for every $\ucN\in\FMod(A_{o_L})$ the canonical map $$\Hom_{\FMod(A_{o_L})}(\ucN,\ucM)\pfeil \Hom_{\FMod(A_L)}(\ucN[\piinv],\ucM[\piinv])$$
is surjective (and hence bijective). 

\smallskip\noindent
(b) An object $\ucM'\in\FMod(A_{o_L,\pi})$ is called \emph{$A_{o_L,\pi}$-maximal} if for every $\ucN'\in\FMod(A_{o_L,\pi})$ the canonical map $$\Hom_{\FMod(A_{o_L,\pi})}(\ucN',\ucM')\pfeil \Hom_{\FMod(A_{o_L,\pi}[\piinv])}(\ucN'[\piinv],\ucM'[\piinv])$$
is surjective (and hence bijective). 

\smallskip\noindent
(c) Let $\uM\in\FMod(A_L)$. An object $\ucM\in\FMod(A_{o_L})$ is called an \emph{$A_{o_L}$-maximal model} for $\uM$ if $\ucM[\piinv]\cong\uM$ inside $\FMod(A_L)$ (i.e.\ $\ucM$ is a \emph{model} for $\uM$) and if $\ucM$ is $A_{o_L}$-maximal. 
Correspondingly, given $\uM'\in\FMod(A_{o_L,\pi}[\piinv])$, an object $\ucM'\in\FMod(A_{o_L,\pi})$ is called an \emph{$A_{o_L,\pi}$-maximal model} for $\uM'$ if $\ucM'[\piinv]\cong\uM'$ inside $\FMod(A_{o_L,\pi}[\piinv])$ and if $\ucM'$ is $A_{o_L,\pi}$-maximal. 
\end{Definition}

The existence of ($A_{o_L}$- and $A_{o_L,\pi}$-)maximal models has been established in \cite{GardeynSST}. 

\begin{Proposition} [{\cite[Proposition~2.13]{GardeynSST}}] \label{Prop1.24}
Let $\uM\in \FMod(A_L)$. Then the following assertions hold:\begin{enumerate}
\item $\uM$ admits an $A_{o_L}$-maximal model, which is unique up to unique isomorphism.
\item If a model $\ucM\in\FMod(A_{o_L})$ of $\uM$ is good in the weak sense of Definition~\ref{Def-goodFMod}, then it is $A_{o_L}$-maximal.
\end{enumerate}
\end{Proposition}

The next proposition is a variant of Gardeyn's theory of maximal models.

\begin{Proposition}\label{Prop1.28}
The following assertions hold: 
\begin{enumerate} 
\item Every $\uM\in\FMod(A_{o_L,\pi}[\piinv])$ admits a maximal model, which is unique up to unique isomorphism.
\item If $\uM\in\FMod(A_L)$ is given and if $\ucM\in\FMod(A_{o_L})$ is an $A_{o_L}$-maximal model of $\uM$ then $\ucM\otimes_{A_{o_L}}A_{o_L,\pi}\in\FMod(A_{o_L,\pi})$ is an $A_{o_L,\pi}$-maximal model of $\uM\otimes_{A_L}A_{o_L,\pi}[\piinv]\in\FMod(A_{o_L,\pi}[\piinv])$.
\item Let $\uM\in \FMod(A_{o_L,\pi}[\piinv])$ and let $\ucM\in\FMod(A_{o_L,\pi})$ be a model of $\uM$. If $\ucM$ is a good model in the weak sense of Definition~\ref{Def-goodFMod}, then it is $A_{o_L,\pi}$-maximal.
\end{enumerate}
\end{Proposition}

\begin{proof}
For (i) (resp.\ (ii); resp.\ (iii)), see \cite{GardeynSST}, 3.3(i) (resp. 3.4(i); resp.\ 2.13(ii)). 
Note that strictly speaking Gardeyn proves these statements for the rings $\Gamma(\fA(\infty),\cO_{\fA(\infty)})$ instead of $A_{o_L,\pi}[\piinv]$ and $\Gamma(\fA(\infty),\cO_{\fA(\infty)})\cap A_{o_L,\pi}$ instead of $A_{o_L,\pi}$. His arguments carry over literally to our rings.
\end{proof}

We may conclude:

\begin{Proposition}\label{GoodModelFModule}
In the weak sense of Definition~\ref{Def-goodFMod} a Frobenius $A_L$-module $\uM$ admits a good model over $A_{o_L}$ if and only if $\uM\otimes_{A_L}A_{o_L,\pi}[\piinv]\in\FMod(A_{o_L,\pi}[\piinv])$ admits a good model over $A_{o_L,\pi}$. If this is the case, the functor $(\ucM,\alpha)\mapsto (\ucM\otimes_{A_{o_L}}A_{o_L,\pi},\,\alpha\otimes\id_{A_{o_L,\pi}[\piinv]})$ is an equivalence of categories between the good models of $\uM$ and the good models of $\uM\otimes_{A_L}A_{o_L,\pi}[\piinv]$.  
\end{Proposition}

\begin{proof}
First suppose that $\uM$ admits a good model $\ucM\in\FMod(A_{o_L})$. It follows that $\ucM$ is an $A_{o_L}$-maximal model of $\uM$. Furthermore, its image $\ucM\otimes_{A_{o_L}}A_{o_L,\pi}$ inside $\FMod(A_{o_L,\pi})$ is an $A_{o_L,\pi}$-maximal model of $\uM\otimes_{A_L}A_{o_L,\pi}[\piinv]$. Since the reduction of $\ucM$ is canonically isomorphic to the reduction of $\ucM\otimes_{A_{o_L}}A_{o_L,\pi}$ by Proposition~\ref{ExpAlgGoodRed}, it follows that the latter is a good model. 

Conversely, suppose that $\uM\otimes_{A_L}A_{o_L,\pi}[\piinv]$ admits a good model $\underline{\widehat\cM}\in\FMod(A_{o_L,\pi})$. Necessarily $\underline{\widehat\cM}$ is a maximal model by Proposition~\ref{Prop1.28}(iii). We know that there is an $A_{o_L}$-maximal model $\ucM\in\FMod(A_{o_L})$ of $\uM$ such that $\ucM\otimes_{A_{o_L}}A_{o_L,\pi}\cong \underline{\widehat\cM}$, and that the reduction of $\underline{\widehat\cM}$ is canonically isomorphic to the reduction of $\ucM$ by Propositions~\ref{Prop1.24}, \ref{Prop1.28}(ii) and \ref{ExpAlgGoodRed}. Since $\underline{\widehat\cM}$ is a good model, so is $\ucM$. This proves the first statement and it also proves essential surjectivity of the functor.

To prove full faithfulness let $(\ucM,\alpha)$ and $(\ucM',\alpha')$ be good models of $\uM$ and let $\hat\beta:\ucM\otimes_{A_{o_L}}A_{o_L,\pi}\isoto\ucM'\otimes_{A_{o_L}}A_{o_L,\pi}$ be an isomorphism in $\FMod(A_{o_L,\pi})$ satisfying $\alpha'\otimes\id=\hat\beta\circ(\alpha\otimes\id)$. Since $A_{o_L}=A_L\cap A_{o_L,\pi}$ inside $A_{o_L,\pi}[\piinv]$, we can recover $\cM$ as $\cM=\alpha(M)\cap\cM\otimes_{A_{o_L}}A_{o_L,\pi}$. This implies $\hat\beta(\cM)=\cM'$ and $\beta:=\hat\beta|_{\cM}$ is the desired isomorphism $\beta\:\ucM\isoto\ucM$ with $\alpha'=\beta\circ\alpha$. This proves full faithfulness.
\end{proof}

For Anderson $A$-motives Proposition~\ref{GoodModelFModule} and Theorem~\ref{ThmGoodModels} imply the following

\begin{Corollary}\label{AlgGoodMod}
Let ${\uM}$ be an Anderson $A$-motive over $L$. Then in the strong sense of Definition~\ref{Def-good}, $\uM$ admits a good model $\ucM$ if and only if the associated analytic Anderson $A(1)$-motive $\uM\otimes_{A_L}A_{o_L,\pi}[\piinv]$ admits a good model $\ucM'$. If this is the case, the functor $(\ucM,\alpha)\mapsto (\ucM\otimes_{A_{o_L}}A_{o_L,\pi},\,\alpha\otimes\id_{A_{o_L,\pi}[\piinv]})$ is an equivalence of categories between the good models of $\uM$ and the good models of $\uM\otimes_{A_L}A_{o_L,\pi}[\piinv]$.\qed
\end{Corollary}

This corollary together with Theorem~\ref{GoodRedCrit} and Corollary~\ref{GoodModel-LocalShtuka} implies the following criterion for good reduction of Anderson $A$-motives, which can be regarded as an analog of the reduction criteria for abelian varieties of Grothendieck~\cite[Proposition~IX.5.13]{SGA7} and de Jong~\cite[2.5]{dJ98}.

\begin{Corollary}\label{CorGoodRedCrit}
Let $\uM$ be an Anderson $A$-motive over $L$ such that $\coker(F_{\uM})$ is annihilated by $\fJ^d$ for some $d$. Then the following assertions are equivalent:
\begin{enumerate}
\item $\uM$ admits a good model $(\ucM,\alpha)$ in the strong sense of Definition~\ref{Def-good}, i.e.\ there is an object $\ucM\in \FMod(A_{o_L})$ such that $\coker(F_{\ucM})$ is a finite free $o_L$-module and is annihilated by $\fJ^d$, together with an isomorphism $\alpha\:\uM\isoto\ucM[\piinv]$ inside $\FMod(A_L)$;
\item There is 
an effective local shtuka $\uMhut$ at $\eps$ over $o_L$ such that $\coker(F_{\hat{M}})$ is annihilated by $\fJ^d$, and an isomorphism $\uM \otimes_{A_L} A_{o_L,\epspi}[\piinv]\cong \uMhut[\piinv]$
inside $\FMod(A_{o_L,\epspi}[\piinv])$.
\end{enumerate}
Moreover, there is an equivalence of categories 
\begin{eqnarray*}
\left\{\begin{array}{l} \text{good models $(\ucM,\alpha)$ of }\uM\text{ in the}\\ \text{sense of Definitions~\ref{Def-good} and \ref{Def-goodFMod}}\end{array}\right\} & \stackrel{\sim}{\longleftrightarrow} &
\left\{\begin{array}{l}
\text{pairs }(\uMhut,f)\text{ consisting of}\\
\text{\textbullet\ a local shtuka }\uMhut\text{ at }\eps\text{ over }o_L,\text{ and}\\
\text{\textbullet\ an isomorphism in }\FMod(A_{o_L,\epspi}[\piinv])\\
\quad f\: \uM\otimes_{A_L} A_{o_L,\epspi}[\piinv]\isoto \uMhut[\piinv]\end{array}\right\} \\[2mm]
(\ucM,\alpha) & \longmapsto & (\ucM,\alpha)\otimes_{A_{o_L}}A_{o_L,\epspi}\,,
\end{eqnarray*}
where on the right-hand side a \emph{morphism} of pairs $\hat\beta\:(\uMhut,f)\isoto(\uMhut',f')$ is defined to be an isomorphism of local shtukas $\hat\beta\:\uMhut\isoto\uMhut'$ satisfying $f'=\hat\beta\circ f$. \qed
\end{Corollary}

\vfill

\begin{minipage}[t]{0.5\linewidth}
\noindent
Urs Hartl\\
Universit\"at M\"unster\\
Mathematisches Institut \\
Einsteinstr.~62\\
D -- 48149 M\"unster
\\ Germany
\\[1mm]
\href{http://www.math.uni-muenster.de/u/urs.hartl/index.html.en}{www.math.uni-muenster.de/u/urs.hartl/}
\end{minipage}
\begin{minipage}[t]{0.45\linewidth}
\noindent
Simon H\"usken\\
Universit\"at M\"unster\\
Mathematisches Institut \\
Einsteinstr.~62\\
D -- 48149 M\"unster
\\ Germany
\\[1mm]
\end{minipage}

\end{document}